\documentclass[a4paper, 11pt, twoside, leqno]{article}

\usepackage[T1]{fontenc}
\usepackage[utf8x]{inputenc}
\usepackage[english]{babel}
\usepackage{lmodern}               % Latin Modern font
\usepackage{amsmath,amsthm,amssymb}
\usepackage{mathrsfs,esint,bbm,stmaryrd}
\usepackage{enumitem,graphicx,xcolor,subcaption}
\usepackage[textwidth=16cm,centering,headheight=24pt]{geometry}
\usepackage{authblk}
\newtheorem{theorem}{Theorem}           % Bold title, italic text

\newtheorem{lemma}[theorem]{Lemma}
\newtheorem{prop}[theorem]{Proposition}

\theoremstyle{definition}              % Bold title, roman text

\theoremstyle{remark}                  % Italic title, roman text
\newtheorem{step}{Step}

\newtheorem{remark}{Remark}

\DeclareMathOperator{\dist}{dist}                                   % distance
\DeclareMathOperator{\BV}{BV}
\DeclareMathOperator{\SBV}{SBV}

\DeclareMathOperator{\Aut}{Aut}
\DeclareMathOperator{\inter}{int}

\newcommand{\abs}[1]{\left| #1 \right|}                             % absolute value
\newcommand{\norm}[1]{\left\| #1 \right\|}                          % norm
\DeclareMathAlphabet{\mathpzc}{OT1}{pzc}{m}{it}

\newcommand{\D}{\mathrm{D}}       % Roman

\renewcommand{\d}{\mathrm{d}}

\newcommand{\N}{\mathbb{N}}       % Blackboard bold
\newcommand{\R}{\mathbb{R}}
\newcommand{\Z}{\mathbb{Z}}

\renewcommand{\SS}{\mathbb{S}}
\newcommand{\G}{\mathbf{G}} % Bold

\newcommand{\NN}{\mathscr{N}}     % Mathscr 

\newcommand{\GG}{\mathscr{G}}
\renewcommand{\H}{\mathscr{H}}
\renewcommand{\L}{\mathscr{L}}
\newcommand{\eps}{\varepsilon}

\SetSymbolFont{stmry}{bold}{U}{stmry}{m}{n}  % Bold Greek letters

\newcommand{\PR}{\mathbb{R}\mathrm{P}^2}  
\newcommand{\EE}{\mathscr{E}}
\newcommand{\RR}{\varrho}
\newcommand{\X}{\mathscr{X}}

\let\oldv\v

\renewcommand{\v}{\mathbf{v}}

\definecolor{lightblue}{rgb}{0.22,0.45,0.70}   % light blue
\definecolor{darkgray}{gray}{0.4}    % dark grey
\definecolor{lightgray}{gray}{0.8}

\title{Lifting for manifold-valued maps of bounded variation}

 \author{
  Giacomo Canevari and Giandomenico Orlandi\thanks{
   Dipartimento di Informatica --- Universit\`a di Verona,
   Strada le Grazie 15, 37134 Verona, Italy. \\
   \emph{E-mail addresses}: \texttt{giacomo.canevari@univr.it},
   \texttt{giandomenico.orlandi@univr.it}
  }
 }

\date{\today}

\begin{document}

\maketitle

\begin{abstract}
 Let~$\NN$ be a smooth, compact, connected Riemannian manifold without boundary.
 Let $\EE\to\NN$ be the Riemannian universal covering of~$\NN$.
 For any bounded, smooth domain~$\Omega\subseteq\R^d$ and 
 any~$u\in\BV(\Omega, \, \NN)$, we show that
 $u$ has a lifting~$v\in\BV(\Omega, \, \EE)$.
 Our result proves a conjecture by Bethuel and Chiron.
\end{abstract}

\section{Introduction}

Let~$\NN$ be a smooth, compact, connected Riemannian manifold without boundary.
% By Nash embedding theorem,  we may identify~$\NN$
% with a submanifold of a Euclidean space, 
% $\NN\subseteq\R^m$, without loss of generality.
Let 
\[
 \pi\colon\EE\to\NN
\]
be the (smooth) universal covering of~$\NN$. 
We endow~$\EE$ with the pull-back metric, 
%~$\pi^*(h_{\NN})$, $h_{\NN}$ being the metric of~$\NN$, 
so that~$\pi$ is a local isometry.
Given a bounded, smooth domain~$\Omega\subseteq\R^d$
and measurable maps~$u\colon\Omega\to\NN$,
$v\colon\Omega\to\EE$, we say that~$v$
is a \emph{lifting} for~$u$ if $\pi\circ v = u$ a.e. on~$\Omega$.
We are interested in the 

\medskip
\textbf{Lifting problem.} 
Given a \emph{regular} map~$u\colon\Omega\to\NN$, is there
a lifting~$v\colon\Omega\to\EE$ of~$u$ that is as regular as~$u$?

\medskip
Of course, the answer %to the lifting problem 
depends on what we mean precisely by ``regular''.
If~$u$ is of class~$C^k$ (with~$k=0, \, 1, \, \ldots, \, \infty$)
and~$\Omega$ is simply connected, then
the lifting problem has a positive answer.
% --- this is a classical result in topology. 
If other notions of regularity --- for instance, Sobolev regularity 
--- are considered, the problem may be more 
delicate. The lifting problem for non-continuous maps has been studied 
first when~$\NN$ is the unit circle, $\NN=\SS^1$,
in connection with the Ginzburg-Landay theory of superconductivity.
In this case, $\EE=\R$ and the covering map~$\pi\colon\R\to\SS^1$
is given by~$\pi(\theta) = \exp(i\theta)$. The study of this case was initiated
in~\cite{BethuelZheng, BethuelDemengel} and culminated with 
the work by Bourgain, Brezis and Mironescu~\cite{BourgainBrezisMironescu2005}, 
who gave a complete answer to the lifting problem when
$u\in W^{s,p}(\Omega, \, \SS^1)$, $s>0$, $1< p <+\infty$. 
Their results have been extended to the Besov setting
by Mironescu, Russ and Sire~\cite{MironescuRussSire}.
Another particular instance of the lifting problem
is the case when~$\NN$ is the real projective plane, $\NN=\PR$,
which is obtained from the $2$-dimensional sphere~$\SS^2$
by identifying pairs of antipodal points. 
The covering space~$\EE$ is then the % unit, bi-dimensional 
sphere~$\SS^2$ and $\pi\colon\SS^2\to\PR$ is the natural projection.
% \cite{BallZarnescu, BallBedford, Bedford, Mucci-DCDS, IgnatLamy}.
$\PR$-valued maps and their lifting have a physical interpretation 
e.g. in materials science, as they serve as models for a class 
of materials known as (uniaxial) nematic liquid crystals 
--- see e.g.~\cite{BallZarnescu, BallBedford} for more details.
The lifting problem for $\PR$-valued maps, in the context of Sobolev~$W^{1,p}$-spaces,
has been studied e.g. by Ball and Zarnescu~\cite{BallZarnescu} 
and Mucci~\cite{Mucci-DCDS}.

For more general target manifolds~$\NN$,
the lifting problem in Sobolev spaces~$W^{s,p}(\Omega, \, \NN)$
was studied Bethuel and Chiron~\cite{BethuelChiron},
and only very recently it has been completely
settled by Mironescu and Van Schaftingen~\cite{MironescuVanSchaftingen}.
Among other results, Bethuel and Chiron proved that,
if~$\Omega$ is simply connected and~$p\geq 2$, 
then every map~$u\in W^{1,p}(\Omega, \, \NN)$
has a lifting~$v\in W^{1,p}(\Omega, \, \EE)$. However, there exist 
maps that belong to~$W^{1,p}(\Omega, \, \NN)$ for any~$p<2$,
and yet have no lifting in~$W^{1,p}(\Omega, \, \EE)$
--- for instance, we can take~$\NN=\SS^1$, $\Omega$ the unit disk in~$\R^2$,
and~$u(x) := x/\abs{x}$.
Bethuel and Chiron raised the conjecture~\cite[Remark~1]{BethuelChiron} that 
any map~$u\in W^{1,p}(\Omega, \, \NN)$, with~$p\geq 1$, has a lifting 
of bounded variation~(BV).

In this paper, we consider the lifting problem when $u$ is a BV-map.
Previous works showed that the lifting problem for~$u\in\BV(\Omega, \, \NN)$
has a positive answer in case~$\NN=\SS^1$
(Giaquinta, Modica and Sou\oldv{c}ek~\cite[Corollary~1 in Section~6.2.2]
{GiaquintaModicaSoucek-II},
Davila and Ignat~\cite{DavilaIgnat}, Ignat~\cite{Ignat-Lifting}),
$\NN=\R\mathrm{P}^k$ (Bedford~\cite{Bedford},
% in case~$u\in\SBV(\Omega, \, \PR)$
Ignat and Lamy~\cite{IgnatLamy})
and more generally, if the fundamental group of~$\NN$,
$\pi_1(\NN)$, is abelian~\cite{CO1}. 
The aim of this paper is to prove a lifting result
for maps $u\in\BV(\Omega, \, \NN)$ \emph{without}
assuming that~$\pi_1(\NN)$ is abelian.
Examples of closed manifolds with non-abelian fundamental group
are obtained by taking the quotient of~$\mathrm{SO}(3)$,
the set of rotations of~$\R^3$, by the symmetry group of a regular, convex polyhedron.
% which is non-abelian. 
The elements of this quotient space
describe the possible orientations of the given polyhedron in~$\R^3$.
Manifolds of this form appear in
variational problems, arising from applications of different kinds.
For instance, in material science, they 
appear in models for ordered materials,
such as biaxial nematics %or the superfluid phases of~$^3$He
(see e.g.~\cite{Mermin}). In numerical analysis,
they are found in Ginzburg-Landau
functionals with applications to mesh generation, via the so-called 
cross-field algorithms (see e.g.~\cite{CrossFields}). 

\paragraph{Setting.}
% Before we state the main result of this paper,
% let us spend a few words on the setting.
By Nash's theorem~\cite{Nash56}, we can embed isometrically
both~$\NN$ and~$\EE$ into Euclidean spaces, $\NN\subseteq\R^m$,
$\EE\subseteq\R^{\ell}$. Moreover, since~$\NN$, $\EE$
are complete Riemannian manifolds, we can choose the embeddings so
that the images of~$\NN$, $\EE$ are \emph{closed} subsets of~$\R^m$, $\R^{\ell}$,
respectively~\cite{Muller-ClosedEmbeddings}.
From now on, we will identify~$\NN$, $\EE$
with their closed Euclidean embeddings.
Given an open set~$\Omega\subseteq\R^d$,
we define $\BV(\Omega, \, \NN)$ as the set of maps $u\in\BV(\Omega, \, \R^m)$ 
that satisfy the pointwise constraint~$u(x)\in\NN$
for a.e.~$x\in\Omega$. We also define $\SBV(\Omega, \, \NN)$ 
as the set of maps~$u\in\BV(\Omega, \, \NN)$
such that the distributional derivative~$\D u$
(taken in the sense of~$\BV(\Omega, \, \R^m)$) has no Cantor part.
We define $\BV(\Omega, \, \EE)$, $\SBV(\Omega, \, \EE)$ in a similar fashion.
We write $\abs{\mu}\!(\Omega)$ to denote the total variation
of a vector-valued Radon measure~$\mu$ on~$\Omega$.

\begin{theorem} \label{th:main}
 Let~$\NN\subseteq\R^m$ be a smooth, compact,
 connected manifold without boundary. 
 Let~$\Omega\subseteq\R^d$ be a smooth, bounded domain with~$d\geq 1$.
 Then, any~$u\in\BV(\Omega, \, \NN)$ 
 has a lifting~$v\in\BV(\Omega, \, \EE)$ that satisfies
 \[
  \norm{v}_{L^1(\Omega)} \leq C_{\Omega, \, \NN}
   \left(\abs{\D u}\!(\Omega) + 1\right) \! , \qquad
  \abs{\D v}\!(\Omega) \leq C_{\Omega, \, \NN}\abs{\D u}\!(\Omega),
 \]
 where the constant~$C_{\Omega, \, \NN}$ depends only on~$\Omega$ 
 and (the given Euclidean embbeding of)~$\NN$.
 Moreover, if~$u\in\SBV(\Omega, \, \NN)$ 
 and~$v\in\BV(\Omega, \, \EE)$ is a lifting of~$u$, then
 $v\in\SBV(\Omega, \, \EE)$. 
\end{theorem}

Theorem~\ref{th:main} implies, in particular,
that a map~$u\in W^{1,p}(\Omega, \, \NN)$ with~$p \geq 1$ 
has a lifting~$v\in\SBV(\Omega, \, \EE)$, thus proving
Bethuel and Chiron's conjecture in~\cite{BethuelChiron}.

The proof of Theorem~\ref{th:main} relies on ideas
from~\cite[Theorem~3]{CO1}. However, the results of~\cite{CO1}
were formulated in terms of flat chains with coefficients
in the group~$\pi_1(\NN)$. We do \emph{not} take this point of view here,
because the theory of flat chains requires the 
coefficient group~$\pi_1(\NN)$ to be abelian.
Nevertheless, the global structure of the proof
is similar to that of~\cite{CO1}:
we approximate a given map~$u\colon\Omega\to\NN$
with \emph{piecewise-affine} maps~$u_j\colon\Omega\to\R^{m}$;
we project the~$u_j$'s onto~$\NN$, so to define maps 
$\Omega\to\NN$ with polyhedral singularities; 
we lift the re-projected maps to~$v_j\colon\Omega\to\EE$,
by means of topological arguments;
and finally, we pass to the limit, thus obtaining 
a lifting~$v\colon\Omega\to\EE$ of~$u$. Our approach is, 
by its nature, extrinsic --- that is, it depends upon 
the choice of embeddings.

The paper is organised as follows. In Section~\ref{sect:prelim},
we revisit the construction of a locally Lipschitz retraction 
$\R^m\setminus\X\to\NN$, where~$\X$ is a lower-dimensional, 
compact subset of~$\R^m$ (Section~\ref{sect:X}), and recall some topological
properties of the covering~$\pi$ (Section~\ref{sect:aut}).
Section~\ref{sect:affine} contains the 
core of the proof: we construct a lifting for a particular class of
$\NN$-valued maps, those that are obtained from piecewise-affine maps
% $\Omega\to\R^m$ 
by projection onto~$\NN$. We complete the proof of 
Theorem~\ref{th:main} in Section~\ref{sect:main}.

\numberwithin{equation}{section}
\numberwithin{definition}{section}
\numberwithin{theorem}{section}
\numberwithin{remark}{section}
\numberwithin{example}{section}

\section{Preliminaries}
\label{sect:prelim}

\subsection{Projecting onto~$\NN$}
\label{sect:X}

As in~\cite{CO1}, our arguments rely on the following topological property
(see e.g. \cite[Lemma~6.1]{HardtLin-Minimizing}, 
\cite[Proposition~2.1]{BousquetPonceVanSchaftingen-I}, 
\cite[Lemma~4.5]{Hopper}).
We recall that, given a topological space~$A$ and a subset~$B\subseteq A$,
a retraction~$\RR\colon A\to B$ is a continuous map such that
$\RR(z) = z$ for any~$z\in B$. 

\begin{prop}\label{prop:X}
 Let~$\NN$ be a smooth, compact, connected submanifold of~$\R^m$,
 without boundary. Let~$M > 0$ be such that $\NN$ is contained
 in the interior of cube $Q^m_M := [-M, \, M]^{m}$.
 Then, there exist a closed set~$\X\subseteq Q^m_M\setminus\NN$ 
 and a locally Lipschitz retraction
 $\RR\colon Q^m_M\setminus\X\to\NN$ with the following properties.
 \begin{enumerate}[label=(\roman*)]
  \item $\X$ is a finite union of
  polyhedra of dimension~$m-2$ at most.
  \item $\RR$ is smooth in a neighbourhood of~$\NN$.
  \item There exists a constant~$C_0>0$ such that
  \[
   \abs{\nabla\RR(z)} \leq C_0 \dist^{-1}(z, \, \X)
  \]
  for a.e.~$z\in Q^m_M\setminus\X$.
  \item There exists a constant~$C_1>0$ such that,
  if~$\gamma\colon [0, \, 1]\to Q^m_M\setminus\X$
  is an injective, Lipschitz map that parametrises a
  straight line segment, then
  \[
   \int_0^1 \abs{(\RR\circ\gamma)^\prime(t)} \d t \leq C_1.
  \]
 \end{enumerate}
\end{prop}

% This result, or variants thereof, was proved in
% \cite[Lemma~6.1]{HardtLin-Minimizing}, 
% \cite[Proposition~2.1]{BousquetPonceVanSchaftingen}, \cite[Lemma~4.5]{Hopper}.

As mentioned above, there are several references that prove the 
existence of~$\X$ and~$\RR$ satisfying Properties~(i)--(iii). 
However, we have not been able to find a reference for Property~(iv)
(although it is, to some extent, reminiscent of the Deformation Theorem 
for integral currents, see e.g. \cite[Theorem~4.2.9]{Federer}).
Since Property~(iv) is crucial for us, we provide 
a proof below. Given an integer~$q\geq 1$, we define 
the grid~$\GG$ on~$Q^m_M$ of size~$M/q$
as the collection of cubes
\begin{equation} \label{grid}
 \GG := \left\{ \frac{Mz}{q} 
  + \left[0, \, \frac{M}{q}\right]^{m} 
  \colon z\in\Z^{m}\cap[- q, \, q - 1]^{m} \right\} \! .
\end{equation}
For~$j\in\{0, \, 1,  \, \ldots, \, m\}$, 
we denote by~$\GG_j$ the collection of the (closed)~$j$-faces of cubes in~$\GG$. 
We define the~$j$-skeleton of~$\GG$ as~$R_j := \cup_{K\in\GG_j} K$.
We define the dual grid to~$\GG$ as 
\[
 \GG^\prime := \left\{ \left(\frac{M}{2q}, \, \frac{M}{2q}, \, \ldots, \, \frac{M}{2q}\right) 
  + \frac{Mz}{q} 
  + \left[0, \, \frac{M}{q}\right]^{m} 
  \colon z\in\Z^{m}\cap[- q - 1, \, q - 1]^{m} \right\} \! .
\]
We denote by~$R^\prime_j$ the $j$-skeleton of~$\GG^\prime$.

\begin{proof}[Proof of Proposition~\ref{prop:X}]
 Given a $j$-dimensional cube~$K\in\G_j$ 
 of centre~$\bar{z}$, we denote by
 $\xi_K\colon K\setminus\{\bar{z}\}\to\partial K$ 
 the radial retraction onto its boundary.
 If we rotate and dilate~$K$ so to 
 have~$K = \{z\in \bar{z} + [-1 , \, 1]^m\colon 
 z_{j+1} = \bar{z}_{j+1}, \, \ldots, \, z_m = \bar{z}_{m}\}$, then
 \[
  \xi_{K} (z) := \frac{z - \bar{z}}{\max_{1 \leq i \leq j} 
  \abs{z_i - \bar{z}_{i}}} 
  \qquad \textrm{for } z\in K\setminus\{\bar{z}\}.
 \]
 The map~$\xi_K$ has the following properties:
 \begin{enumerate}[label=(\alph*)]
  \item $\xi_K$ is locally Lipschitz, 
  and~$\abs{\nabla\xi_K(z)}\leq C\abs{z-\bar{z}}^{-1}$.
  \item If~$L\subseteq K\setminus\{\bar{z}\}$
  is a straight line segment, then~$\xi_K(L)$ is
  a finite union of segments,
  each one contained in a $(j-1)$-face of~$\partial K$.
  \item If~$L\subseteq K\setminus\{\bar{z}\}$
  is a straight line segment, then for all but at most one~$z\in\partial K$
  we have $\H^0(L\cap \xi_K^{-1}(z)) \leq 1$.
  Indeed, for any~$z\in\partial K$, the inverse image~$\xi_K^{-1}$
  is contained in a straight line~$L_z$; 
  if~$L_z\cap L$ contains more than one point,
  then $L_z\supseteq L$ and~$\xi_K(L) = \{z\}$. 
 \end{enumerate}
  
 Since~$\NN$ is compact and smooth, for $r>0$ small enough
 the $r$-neighbourhood of~$\NN$ retracts smoothly onto~$\NN$
 (by nearest-point projection onto~$\NN$, for instance).
 Let us fix an integer~$q\geq 1$, and let us consider the grid~$\GG$
 of size~$M/q$, defined by~\eqref{grid}. If~$q$ is large enough, 
 there exists a set~$W\subseteq\ Q^m_M$ that is a finite union 
 of cubes of~$\GG$, contains~$\NN$ in its interior,
 and retracts smoothly onto~$\NN$. Let~$\RR_W\colon W\to\NN$
 be a smooth retraction. We extend~$\RR_W$ to a Lipschitz
 map~$R_1\cup W\to\NN$, still denoted~$\RR_W$ for simplicity.
 To do so, we first take an arbitrary extension
 $R_0\cup W\to\NN$ of~$\RR_W$. Then, for any ($1$-dimensional) 
 edge~$E$ of~$\GG$ with~$E\not\subseteq W$, we 
 define~$\RR_W\colon E\to\NN$ as a Lipschitz path that
 joins the values at the endpoints. Such a path exists, 
 because~$\NN$ is connected.
 
 We construct a locally Lipschitz retraction
 $\sigma_2\colon (R_2\setminus R^\prime_{m-2})\cup (R_2\cap W) \to
 R_1\cup (R_2\cap W)$, in the following way.
 If~$K\in\GG_2$ is contained in~$W$, then we must
 define~$\sigma_2$ to be the identity on~$K$.
 Take a $2$-dimensional cube~$K\in\GG_2$
 that is not contained in~$W$. We observe that  
 $K\cap R^\prime_{m - 2}$ is exactly the centre of~$K$, and we define
 $\sigma_2(z) := \xi_K(z)$ for~$z\in K\setminus R^\prime_{m-2}$.
 If~$K_1, \, K_2\in\GG_2$ share a common edge~$E\in\GG_1$, then
 $\xi_{K_1}(z) = \xi_{K_2}(z) = z$ for any~$z\in E$,
 so the definition of~$\sigma_2$ is consistent. Moreover,
 $\sigma_2$ is locally Lipschitz and
 \begin{equation} \label{X1}
  \abs{\nabla\sigma_2(z)} \leq C \dist^{-1}(z, \, R^{\prime}_{m-2})
  \qquad \textrm{for a.e. } z\in R_2\setminus R^\prime_{m-2},
 \end{equation}
 by~(a) above.
 
 We contruct now a locally Lipschitz retraction
 $\sigma_3\colon (R_3\setminus R^\prime_{m-2})\cup (R_3\cap W)\to
 R_1 \cup (R_3\cap W)$, in a similar way.
 Given a cube~$K\in\GG_3$ that is not contained in~$W$, we observe that  
 $\xi^{-1}_K (\partial K\cap R^\prime_{m - 2}) = K\setminus R^\prime_{m-2}$,
 and define $\sigma_3(z) := \sigma_2(\xi_K(z))$ 
 for~$z\in K\setminus R^\prime_{m-2}$.
 Again, this definition is consistent.
 Moreover, let~$z_0$ be the centre of the cube~$K$,
 let~$z\in K\setminus\{z_0\}$,
 and let $F$ be a $2$-dimensional face of~$\partial K$, such that
 $\xi_K(z)\in F$. Using the chain rule,
 Property~(a) and~\eqref{X1} above, we deduce that
 \[
  \begin{split}
   \abs{\nabla\sigma_3(z)} &\leq \abs{(\nabla\sigma_2)(\xi_K(z))} \abs{\nabla\xi_K(z)}
   \leq C \dist^{-1}(\xi_K(z), \, R^\prime_{m-2}) \abs{z - z_0}^{-1}.
  \end{split}
 \]
 On the other hand, if~$w$ denotes the projection of~$z$ onto~$R^\prime_{m-2}$,
 we have 
 \[
  \abs{z - z_0} \geq \abs{w - z_0} = \frac{M}{2q} 
  \frac{\dist(z, \, R^\prime_{m-2})}{\dist(\xi_K(z), \, R^\prime_{m-2})}
 \]
 (see Figure~\ref{fig:retraction}). As a result,
 \begin{equation} \label{X2}
  \abs{\nabla\sigma_3(z)} \leq \frac{2Cq}{M} \dist^{-1}(z, \, R^\prime_{m-2}).
 \end{equation}
 \begin{figure}[t]
		\centering
		\includegraphics[height=.3\textheight]{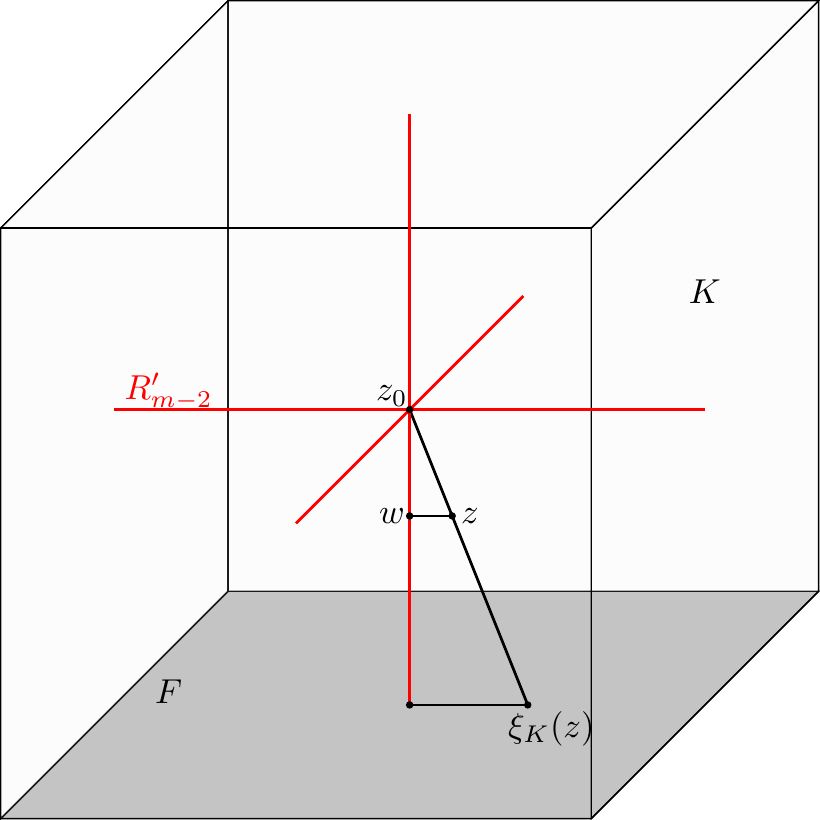}
	\caption{Estimate on~$\abs{z - z_0}$, in case~$m=3$
	(see the proof of Proposition~\ref{prop:X}).}
	\label{fig:retraction}
 \end{figure}
 By induction, we define a sequence of locally Lipschitz retractions
 $\sigma_j\colon (R_j\setminus R^\prime_{m-2})\cup(R_j\cap W)
 \to R_1\cup(R_j\cap W)$,
 for~$j= 4, \, \ldots, \, m$. We take~$\X$ as the closure of 
 $R^{\prime}_{m-2}\cap (Q^m_M\setminus W)$,
 and $\RR := \RR_W\circ\sigma_m$. Properties~(i), (ii)
 is now immediate, while~(iii) follows 
 from~\eqref{X2} by an inductive argument.
 Let~$\gamma\colon [0, \, 1]\to Q^m_M\setminus\X$ be an injective, Lipschitz 
 map that parametrises a straight line segment~$L\subseteq Q^m_M\setminus\X$. 
 By applying Property~(b) iteratively,
 we see that
 \[
  \sigma_m(L) \subseteq (L\cap W) \cup R_1 \cup \bigcup_{i=1}^p L_i,
 \]
 where the $L_i$'s are straight line segments, 
 each one contained in a~$(m-1)$-face of~$\partial W$. 
 By the area formula, we have
 \[
  \begin{split}
   \int_0^1 \abs{(\sigma_m\circ\gamma)^\prime(t)} \d t 
%    \leq \int_{\gamma^{-1}(W)} \abs{\gamma^\prime(t)} \, \d t 
%    + \int_{R_1} \H^0((\sigma_m\circ\gamma)^{-1}(z)) \, \d\H^1(z) 
%    + \sum_{i=1}^p \int_{L_i} \H^0((\sigma_m\circ\gamma)^{-1}(z)) \, \d\H^1(z) 
   \leq \H^1(L\cap W) 
   &+ \int_{R_1} \H^0((\sigma_m\circ\gamma)^{-1}(z)) \, \d\H^1(z) \\
   &+ \sum_{i=1}^p \int_{L_i} 
    \H^0((\sigma_m\circ\gamma)^{-1}(z)) \, \d\H^1(z)
  \end{split}
 \]
 By Properties~(b), (c) and an inductive argument
 we deduce that, for $\H^1$-a.e.~$z\in \cup_i L_i \cup R_1$,
 $\H^0(L\cap\sigma_m^{-1}(z))$ is bounded in terms of~$m$ only.
 Since~$\gamma$ is injective, $\H^0((\sigma_m\circ\gamma)^{-1}(z))$
 is also bounded in terms of~$m$ only. As a result,
 \[
  \begin{split}
   \int_0^1 \abs{(\sigma_m\circ\gamma)^\prime(t)} \d t
   \leq C\left( \mathrm{diam}(W) + \H^1(R_1) 
   + \sum_{K\in\GG_{m-1}, \, K\subseteq\partial W} \mathrm{diam}(K) \right) \!,
  \end{split}
 \]
 where~$\mathrm{diam}$ denotes the diameter. 
 Now~(iv) follows, because $\RR =\RR_W\circ\sigma_m$
 and~$\RR_W$ is Lipschitz.
\end{proof}

Let us choose $M$, $\X$ and~$\RR$ as in Proposition~\ref{prop:X},
once and for all. Let~$\sigma>0$ be a small parameter, such that
$\NN\subseteq (-M + \sigma, \, M - \sigma)^m$.
Let~$B^m_\sigma := \{y\in\R^m\colon \abs{y}<\sigma\}$,
% be the closed ball in~$\R^m$, centred at the origin, of radius~$\sigma$,
and
\begin{equation} \label{Lambda}
 \Lambda := M - \sigma .
\end{equation}
For any~$y\in B^m_\sigma$, the map~$\tilde{\RR}_y \colon z\mapsto\RR(z - y)$
is well defined and locally Lipschitz 
in~$Q^m_\Lambda\setminus(\X+y)$.
Moreover, reducing the value of~$\sigma>0$ if necessary,
the restriction~$\tilde{\RR}_{y|\NN}$ is a small, smooth
perturbation of the identity --- in particular, it is a diffeomorphism.
We define 
% the map $\RR_y\colon Q^m_{\Lambda}\setminus(\X+y)\to\NN$ as
\begin{equation}\label{RR_y}
 \RR_y(z) := \left(\left(\tilde{\RR}_{y|\NN}\right)^{-1} \circ\RR\right)(z - y)
 \qquad \textrm{for } z\in Q^m_\Lambda\setminus(\X+y), \ y\in B^m_\sigma.
\end{equation}
% This map is indeed a locally Lipschitz retraction onto~$\NN$.

\begin{lemma}\label{lemma:projection}
 Let~$\Omega\subseteq\R^d$ be a bounded domain.
 For any~$u\in W^{1,1}(\Omega, \, Q^m_\Lambda)$
 and a.e.~$y\in B^m_\sigma$, the map~$\RR_y\circ u$
 belongs to~$W^{1,1}(\Omega, \, \NN)$.
 Moreover, there exists a constant~$C_\Lambda$ (depending only 
 on~$\NN$, $m$, $\X$, $\RR$, $\sigma$ and~$\Lambda$) such that
 \begin{equation*}
  \int_{B^m_\sigma} \norm{\nabla (\RR_y\circ u)}_{L^1(\Omega)} \, \d y \leq 
  C_\Lambda \norm{\nabla u}_{L^1(\Omega)} \! .
 \end{equation*}
\end{lemma}

This well-know result is based on an argument by
Hardt, Kinderlehrer and Lin \cite[Lem\-ma~2.3]{HKL},
\cite[Theorem~6.2]{HardtLin-Minimizing} 
(a proof of this statement may also be found, e.g.,
in~\cite[Lemma~14]{CO1}).

\subsection{Automorphisms of~$\pi$}
\label{sect:aut}

Let~$\Aut(\pi)$ be the set of smooth isometries~$\varphi\colon\EE\to\EE$ 
such that~$\pi\circ\varphi = \pi$. The set~$\Aut(\pi)$
is a group with respect to the composition of maps. In fact,
$\Aut(\pi)$ is isomorphic to~$\pi_1(\NN, \, z_0)$ for any~$z_0\in\NN$,
although the isomorphism may not be canonical
(it is canonical if and only if~$\pi_1(\NN, \, z_0)$ is abelian).
% The group~$\Aut(\pi)$ acts transitively on each fibre of~$\pi$, that is,
For any $w_1, \, w_2\in\EE$ such that $\pi(w_1) = \pi(w_2)$,
there exists a unique $\varphi\in\Aut(\pi)$ such that 
\begin{equation} \label{trans}
 \varphi(w_1) = w_2
\end{equation}
(see e.g.~\cite[Section~1.3, p.~70]{Hatcher}
or~\cite[Chapter~12]{Lee-Topological}).

\begin{lemma} \label{lemma:cover}
 There exists a compact subset~$E_*\subseteq\EE$ such that 
 \[
  \EE \subseteq \bigcup_{\varphi\in\Aut(\pi)} \varphi(E_*).
 \]
\end{lemma}
\begin{proof}
 Fix base points~$z_1\in\NN$, $w_1\in\pi^{-1}(z_1)\subseteq\EE$.
 For any~$z\in\NN$, choose a minimising 
 geodesic~$\gamma_z\colon[0, \, 1]\to\NN$
 with endpoints~$\gamma_z(0) = z$,
 $\gamma_z(1) = z_1$. Since~$\NN$ is compact, we have
 \begin{equation} \label{cover1}
  R := \sup_{z\in\NN} \int_0^1
  \abs{\gamma_z^\prime(t)} \, \d t < +\infty.
 \end{equation}
 Let $E_*\subseteq\EE$ be the closed geodesic disk of centre~$w_1$
 and radius~$R$. Given~$w\in\EE$, we consider the
 (unique) Lipschitz map~$\tilde{\gamma}\colon [0, \, 1]\to\EE$ 
 such that~$\pi\circ\tilde{\gamma} = \gamma_{\pi(w)}$
 and~$\tilde{\gamma}(0) = w$. We
 have~$\pi(\tilde{\gamma}(1)) = z_1 = \pi(w_1)$
 and hence, due to~\eqref{trans}, there exists a (unique) $\varphi\in\Aut(\pi)$
 such that $\tilde{\gamma}(1) = \varphi(w_1)$. 
 Since~$\pi$ is a local isometry, the geodesic distance 
 between~$\tilde{\gamma}(0) = w$ and~$\tilde{\gamma}(1) = \varphi(w_1)$
 must be less than or equal to~$R$, by~\eqref{cover1}.
 Since~$\varphi$ is an isometry, $w\in\varphi(E_*)$,
 and the lemma follows.
\end{proof}

\section{The case of piecewise-affine maps}
\label{sect:affine}

Towards the proof of Theorem~\ref{th:main}, 
we first construct a lifting for maps of the form~$\RR_y\circ u$,
where~$u$ is piecewise-affine (but not necessarily $\NN$-valued).
For any~$d\geq 1$, let~$Q^d := (-1, \, 1)^d$.
We say that a map~$u\colon Q^d\to\R^m$
is \emph{piecewise-affine} if~$u$ is continuous and 
there exists a triangulation~$\mathcal{T}$ of~$Q^d$
such that, for any simplex~$T$ of~$\mathcal{T}$,
$u_{|T}$ is affine.

\begin{prop} \label{prop:affine}
 Let~$u\colon Q^d\to Q^m_\Lambda$ be
 piecewise-affine. Then, there exist~$y\in B^m_\sigma$ and 
 a lifting~$v\in\BV(Q^d, \, \EE)$
 of~$\RR_y\circ u$ %\in W^{1,1}(Q^d, \, \NN)$ 
 such that
 \[
  \norm{v}_{L^1(Q^d)} \leq 
   C_\Lambda\left(\norm{\nabla u}_{L^1(Q^d)} + 1\right) \! , \qquad 
  \abs{\D v} \! (Q^d) \leq 
   C_\Lambda\norm{\nabla u}_{L^1(Q^d)} \! ,
 \]
 where~$C_\Lambda$ is a constant that depends only 
 on~$d$, $\NN$, $m$, $\X$, $\RR$, $\sigma$ and~$\Lambda$. 
\end{prop}

Before proving Proposition~\ref{prop:affine},
we state some auxiliary results. 
Let~$u\colon Q^d\to Q^m_\Lambda$
be piecewise-affine. 
Let us take a constant~$u_*\in\NN$,
and define $U\colon [0, \, 1]\times Q^d\to Q^m_\Lambda$ by
\begin{equation} \label{affineint}
 U(t, \, x) := (1-t) u(x) + t u_* \qquad 
 \textrm{for } (t, \, x)\in [0, \, 1]\times Q^d.
\end{equation}
Let~$\tau\colon[0, \, 1]\times\R^d\to\R^d$ denote the canonical projection, 
$\tau(t, \, x) := x$. For any~$y\in B^m_\sigma$, we define
\begin{equation} \label{SyTy}
 S_y := (u - y)^{-1}(\X),  \qquad 
 T_y := \tau\left((U - y)^{-1}(\X)\right) \! .
\end{equation}

\begin{lemma} \label{lemma:inverseimage}
 Suppose that~$d\geq 2$. Then,
 for a.e.~$y\in B^m_\sigma$, the sets~$S_y$, $T_y$
 are finite unions of polyhedra, of dimension
 less than or equal to $d-2$, $d-1$ respectively. 
 Moreover,
 \[
  \int_{B^m_\sigma} \H^{d-1}\left(T_y\right) \d y 
  \leq C\int_{Q^d} \abs{u - u_*} \abs{\nabla u} \d\L^d
 \]
 where~$C$ is a positive constant that 
 depends only on~$\NN$, $m$, $\X$, $\RR$.
\end{lemma}

We will prove Lemma~\ref{lemma:inverseimage} with the help of the following 

\begin{lemma}\label{lemma:coarea}
 Let~$\Omega\subseteq\R^d$ be a bounded domain, 
 with~$d\geq 2$.
 Let~$v\colon\R^d\to\R^2$ be an affine map,
 let~$v_*\in\R^2$ be a constant, and let
 $V\colon[0, \, 1]\times\R^d\to\R^2$ be defined by 
 $V(t, \, x) := (1-t)v(x) + tv_*$. Then, 
 \[
  \int_{\R^2} \H^{d-1}\left(\tau(V^{-1}(z))\cap\Omega\right) \d z 
  \leq 2 \int_{\Omega} \abs{v - v_*}\abs{\nabla v} \d\L^d .
 \]
\end{lemma}

Lemma~\ref{lemma:coarea} follows immediately from \cite[Lemma~15]{CO1}.
However, for the convenience of the reader, 
we recall the proof here.

\begin{proof}[Proof of Lemma~\ref{lemma:coarea}]
 We consider the case~$d=2$ first.
 We use~$\nabla$ to denote the gradient in~$\R\times\R^2$,
 with respect to the variables~$(t, \, x)$, and we 
 call~$V^1$, $V^2$ the components of~$V$, i.e. $V = (V^1, \, V^2)$.
 Let~$z\in\R^2$ be a regular value for~$V$.
 Then, $V^{-1}(z)$ is a smooth curve 
 in~$\R\times\R^2$ and the vector 
 field~$(\nabla V^1\times\nabla V^2)/|\nabla V^1\times\nabla V^2|$
 is tangent to~$V^{-1}(z)$. By the area formula,
 \begin{equation} \label{coarea1}
  \H^1\left(\tau(V^{-1}(z))\cap\Omega\right)
  \leq \int_{V^{-1}(z)\cap ([0, \, 1]\times\Omega)}
  \abs{\tau\left(\frac{\nabla V^1\times\nabla V^2}
  {\abs{\nabla V^1\times\nabla V^2}}\right)} \d\H^1
 \end{equation}
 Let~$\mathbf{e} := (1, \, 0, \, 0)\in\R\times\R^2$.
 By the properties of the cross product, we deduce
 \[
  \begin{split}
  \abs{\tau(\nabla V^1\times\nabla V^2)}
  = \abs{\mathbf{e}\times(\nabla V^1\times\nabla V^2)}
  = \abs{(\partial_t V^2)\nabla V^1 - 
  (\partial_t V^1)\nabla V^2}
  \end{split}
 \]
 The $t$-component of~$(\partial_t V^2)\nabla V^1 - 
 (\partial_t V^1)\nabla V^2$ vanishes, so
 \[
  \abs{\tau(\nabla V^1\times\nabla V^2)}
  = \abs{(\partial_t V^2)\nabla_x V^1 - 
  (\partial_t V^1)\nabla_x V^2} 
  \leq 2\abs{\partial_t V} \abs{\nabla_x V}
  \leq 2\abs{v - v_*} \abs{\nabla_x v}
 \]
 Injecting this inequality into~\eqref{coarea1}, we obtain
 \begin{equation*} 
  \H^1\left(\tau(V^{-1}(z))\cap\Omega\right)
  \leq 2 \int_{V^{-1}(z)\cap ([0, \, 1]\times\Omega)}
  \frac{\abs{v - v_*} \abs{\nabla_x v}}
  {\abs{\nabla V^1\times\nabla V^2}} \,\d\H^1.
 \end{equation*}
 We have~$\abs{\nabla V^1\times\nabla V^2}^2 
 = \det(\nabla V (\nabla V)^{\mathsf{T}})$,
 that is, $\abs{\nabla V^1\times\nabla V^2}$
 is the Jacobian of~$V$ (up to a sign). Then, 
 the coarea formula implies
 \begin{equation*} 
  \int_{\R^2} \H^1\left(\tau(V^{-1}(z))\cap\Omega\right) \d z
  \leq 2 \int_{[0, \, 1]\times\Omega}
  \abs{v - v_*} \abs{\nabla_x v} \,\d\L^{d+1},
 \end{equation*}
 which gives the desired estimate when~$d=2$.
 Now, suppose that~$d\geq 3$. Up to a translation,
 we may assume that~$v$ is a linear map. 
 Then, the kernel of~$v$
 contains a~$(d-2)$-linear subspace~$\Pi$,
 and $V^{-1}(z) = (V|_{\Pi^\perp})^{-1}(z)\times\Pi$
 where~$\Pi^\perp$ is the orthogonal complement of~$\Pi$. 
 Therefore, the lemma follows by a slicing argument.
\end{proof}

\begin{proof}[Proof of Lemma~\ref{lemma:inverseimage}]
 By assumption, there exists a triangulation~$\mathcal{T}$ 
 of~$Q^d$ such that, for any simplex~$H$
 of~$\mathcal{T}$ (of arbitrary dimension),
 the restriction $u_{|H}$ is affine. 
 By Proposition~\ref{prop:X}, the set~$\X$ is a finite union of polyhedra,
 say~$K_1$, \ldots, $K_p$, of dimension~$m-2$ at most.
 For any~$y\in B^m_\sigma$, any simplex~$H$ of~$\mathcal{T}$ 
 and any~$i\in\{1, \, \ldots, \, p\}$, 
 $(u-y)^{-1}(K_i)\cap\times H$
 is a polyhedron, hence~$S_y$ is a union of polyhedra.
 To show that~$T_y$ is a union of polyhedra too,
 we assume that~$u_*= 0$, up to a translation.
 For any~$y\in\R^m$ such that~$-y\notin\X$,
 there exists~$\eps>0$ such that
 \[
    T_y\cap H = \bigcup_{i=1}^p 
    \left\{x\in H\colon u(x)\in\tilde{K}_{i,y}\right\}
    \qquad \textrm{where } \
    \tilde{K}_{i,y} := \bigcup_{t \in [0, \, 1-\eps]}\frac{K_i + y}{1 - t}
 \]
 The set~$\tilde{K}_{i,y}$ itself is a polyhedron
 (it is the convex hull of~$(K_i+y) \cup (K_i+y)/\eps$),
 so~$T_i$ is a finite union of polyhedra.
 
 Let~$\Pi_i$ be the affine subspace of~$\R^m$ spanned by~$K_i$.
 For any simplex~$H$ of~$\mathcal{T}$,
 any~$i$ and a.e.~$y\in B^m_\sigma$, the maps~$(u - y)_{|H}$,
 $(U - y)_{|[0, \, 1]\times H}$ are transverse to~$\Pi_i$.
 This follows by Thom's parametric transversality theorem 
 (see e.g.~\cite[Theorem~2.7 p.~79]{Hirsch}).
%  but it can also 
%  be seen by elementary arguments, as follows. Given~$H$ and~$i$,
%  there are two cases: either (i) the image of~$u_{H}$ 
%  spans a plane that is transverse to~$\Pi_i$, or (ii) it does not.
%  In case~(i), $(u - y)_{H}$ is transverse to~$\Pi_i$
%  for any~$y\in B^m_\sigma$. In case~(ii), there exist a linear subspace
%  $\Pi\subsetneq\R^m$ and vectors~$y_1$, $y_2\in\R^m$
%  such that $u(H)\subseteq\Pi + y_1$,
%  $\Pi_i\subseteq\Pi + y_2$. Then
%  $(u - y)(H)\cap\Pi_i=\emptyset$
%  for any~$y\in B^m_\sigma\setminus(\Pi + y_1 - y_2)$,
%  and the claim holds (vacuously) in this case, too.
 By transversality, for any simplex~$H$ in $\mathcal{T}$, any~$i$
 and a.e.~$y\in B^m_\sigma$, we have
 \begin{equation} \label{inverseimage0}
  \dim\left((u-y)^{-1}(\Pi_i)\cap ([0, \, 1]\times H)\right) 
  = \dim(H) + 1 - m + \dim\Pi_i \leq d - 1
 \end{equation}
 (unless the intersection is empty),
 with equality only if~$\dim(H) = d$ and~$\dim(\Pi_i) = m-2$.
 Then, for a.e.~$y\in B^m_\sigma$, $T_y$ has dimension $d-1$ at most.
 In a similar way, we show that $\dim(S_y)\leq d-2$
 for a.e.~$y\in B^m_\sigma$. Moreover, from~\eqref{inverseimage0}
 we deduce
 \begin{equation} \label{inverseimage1}
  \begin{split}
   \H^{d-1}(T_y) \leq
   \sum_{\substack{i\colon\dim\Pi_i = m-2 \\ 
   H\in\mathcal{T}\colon \dim(T) = d}} \
   \H^{d-1}\left(\tau((U-y)^{-1}(\Pi_i))\cap \inter(H)\right) \! ,
  \end{split}
 \end{equation}
 where~$\inter(H)$ denotes the interior of~$H$.
 Now, take~$i$ such that~$\dim\Pi_i = m-2$
 and~$H\in\mathcal{T}$ of dimension~$d$. Let~$\Pi_i^\perp\subseteq\R^m$
 be the orthogonal $2$-plane to~$\Pi_i$, passing through the origin.
 Let~$\zeta$, $\zeta^\perp$ be the orthogonal projections of~$\R^m$
 onto~$\Pi_i$, $\Pi_i^\perp$, respectively. We denote the variable~$y \in\R^m$
 as~$(z, \, z^\perp)\in\Pi_i\times\Pi_i^\perp$. We have
 \begin{equation*}
  \begin{split}
   &\int_{B^m_\sigma} \H^{d-1}\left(\tau((U-y)^{-1}(\Pi_i))\cap\inter(H)\right) \d y \\
   &\qquad\qquad \leq \int_{\zeta(B^m_\sigma)\times\Pi^\perp_i} 
   \H^{d-1}\left(\tau((\zeta^\perp\circ U)^{-1}(z^\perp))\cap\inter(H)\right) 
   \d(z, \, z^\perp).
  \end{split}
 \end{equation*}
 Then, by applying Lemma~\ref{lemma:coarea}
 to the map~$\zeta^\perp\circ U$, we obtain
 \begin{equation*} %\label{inverseimage2}
  \begin{split}
   \int_{B^m_\sigma} \H^{d-1}\left(\tau((U-y)^{-1}(\Pi_i))\cap\inter(H)\right) \d y
   \leq 2\int_{H} \abs{\zeta^\perp\circ u - \zeta^\perp(u_*)}
   \abs{\nabla(\zeta^\perp\circ u)} \d\L^d.
  \end{split}
 \end{equation*}
 Using that $\zeta^\perp$ is $1$-Lipschitz,
 taking the sum over~$i$ and~$H$,
 and applying~\eqref{inverseimage1}, the lemma follows.
\end{proof}

\begin{proof}[Proof of Proposition~\ref{prop:affine}]
 Let us focus on the interesting case~$d\geq 2$;
 we will deal with the case~$d=1$ later.
 We take a constant~$u_*\in\NN$ and define~$U$, $S_y$, $T_y$
 as in~\eqref{affineint}, \eqref{SyTy}.
 By Lemma~\ref{lemma:projection}, Lemma~\ref{lemma:inverseimage}
 and an average argument, we can choose~$y\in B^m_\sigma$ such that
 $S_y$, $T_y$ are polyhedral, of dimension~$d-2$, $d-1$ respectively,
 and
 \begin{equation} \label{smooth1}
  \norm{\nabla (\RR_y\circ u)}_{L^1(Q^d)} + \H^{d-1}(T_y)
  \leq C_{\Lambda} \norm{\nabla u}_{L^1(Q^d)} \! .
 \end{equation}
 (The constant~$C_\Lambda$ depends only on~$d$, $\NN$, $m$,
 $\X$, $\RR$, $\Lambda$.)
%  As above, we let~$\tau\colon[0, \, 1]\times\R^d\to\R^d$
%  denote the canonical projection, $\tau(t, \, x) := x$. 
 We also define
 \begin{equation} \label{Sigmay}
  \Sigma_y := \left\{(t, \, x)\in [0, \, 1]\times Q^d
  \colon \textrm{there exists } s\in [t, \, 1] \textrm{ such that }
  (s, \, x)\in (U - y)^{-1}(\X) \right\}
 \end{equation}
 (see Figure~\ref{fig:Sigmay}). The set~$\Sigma_y$ is closed. 
 We have~$(U - y)^{-1}(\X)\subseteq\Sigma_y$ and hence,
 the map~$\RR_y\circ U$ is well-defined and locally Lipschitz
 on~$([0, \, 1]\times Q^d)\setminus\Sigma_y$.
%  By~\eqref{SyTy}, we also have 
%  \begin{equation} \label{Sigmay_trace}
%   \left(\{0\}\times Q^d\right)\cap\Sigma_y = \{0\}\times T_y.
%  \end{equation}

 \begin{figure}[ht]
	\begin{subfigure}{.65\textwidth}
		\centering
		\includegraphics[height=.3\textheight]{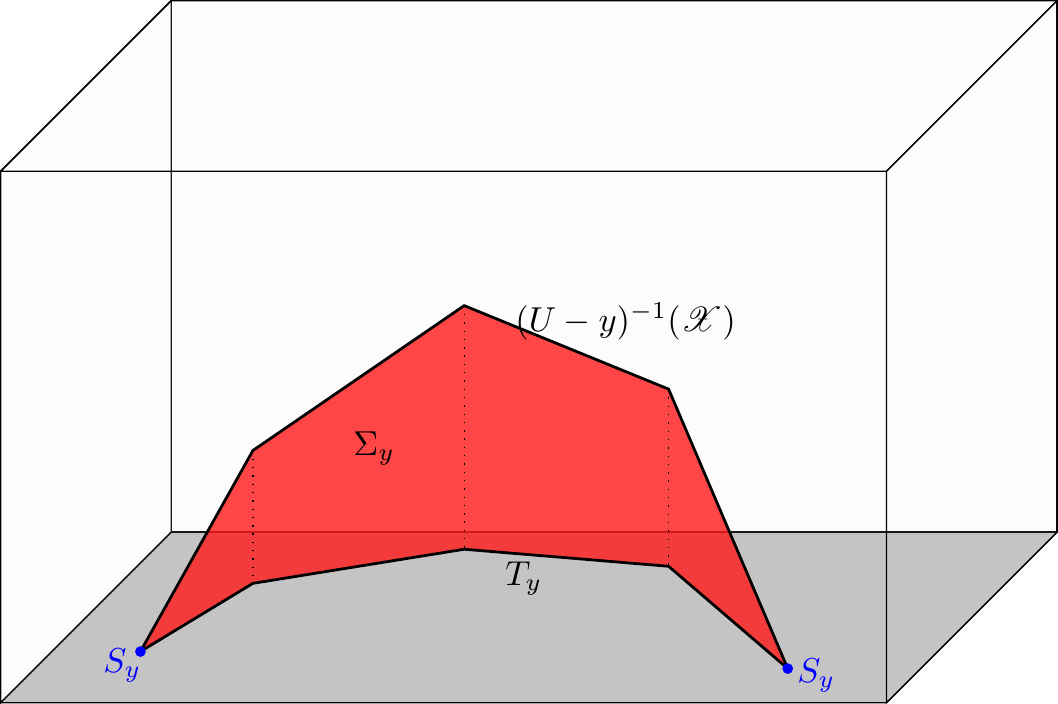}
	\end{subfigure} \hspace{.16cm}
	\begin{subfigure}{.33\textwidth}
		\centering
		\includegraphics[height=.34\textheight]{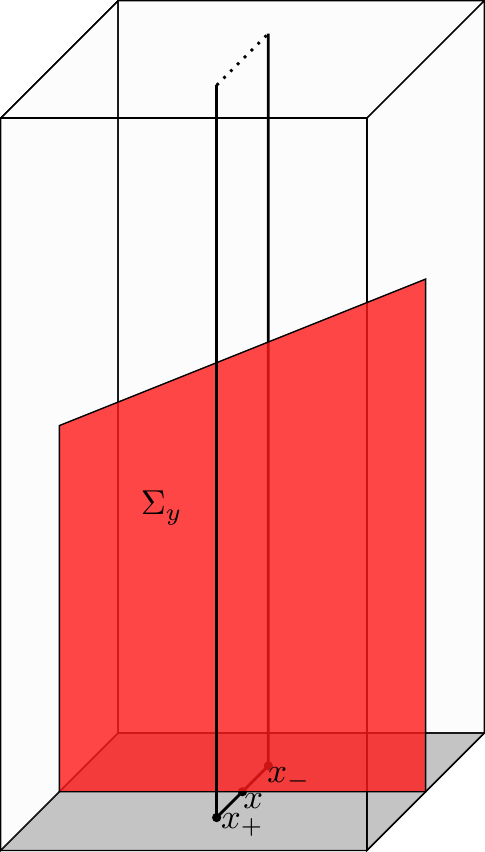}
	\end{subfigure}
	\caption{Left: The set~$\Sigma_y\subseteq([0, \, 1]\times Q^d)$ (in red),
	in case~$Q^d$ is a two-dimensional square (in gray). The complement
	$([0, \, 1]\times Q^d)\setminus\Sigma_y$ retracts by deformation
	onto~$\{1\}\times Q^d$. Right: The path we use in Step~2, 
	for the proof of~\eqref{smooth3}.}
	\label{fig:Sigmay}
 \end{figure}
 
 \setcounter{step}{0}
 \begin{step}[Construction of a lifting]
  We first construct a continuous 
  lifting~$V\colon ([0, \, 1]\times Q^d)\setminus\Sigma_y\to\EE$
  of~$\RR_y\circ U$ restricted 
  to~$([0, \, 1]\times Q^d)\setminus\Sigma_y$.
  A classical result in topology 
  (see e.g. \cite[Proposition~1.33]{Hatcher} 
  or~\cite[Theorem~11.18]{Lee-Topological}) 
  asserts that such a lifting exists
  if and only if, for any  continuous loop
  $\gamma\colon\SS^1\to([0, \, 1]\times Q^d)\setminus\Sigma_y$,
  the composition $\RR_y\circ U\circ\gamma\colon\SS^1\to\NN$
  is homotopic to a constant.
  (It does not matter whether we consider free or based homotopies here,
  because a loop that is freely homotopic to a constant is also homotopic 
  to a constant relative to its base point.)
  Let~$\gamma\colon\SS^1\to([0, \, 1]\times Q^d)\setminus\Sigma_y$
  be a continuous loop. Let~$\xi\colon [0, \, 1]\times\R^d\to [0, \, 1]$,
  $\tau\colon [0, \, 1]\times\R^d\to\R^d$
  be the projections, $\xi(t, \, x) := t$ and
  $\tau(t, \, x):= x$. For any~$\omega\in\SS^1$
  and~$s\geq(\xi\circ\gamma)(\omega)$,
  we have $(s, \, (\tau\circ\gamma)(\omega))\notin (U-y)^{-1}(\X)$
  because of~\eqref{Sigmay}. Then, the map
  $H\colon [0, \, 1]\times\SS^1\to\NN$,
  \[
   H(t, \, \omega) := (\RR_y\circ U) 
   \left( (1 - t)(\xi\circ\gamma)(\omega) + t, \ \,
   (\tau\circ\gamma)(\omega)\right)
   \qquad \textrm{for } (t, \, \omega)\in [0, \, 1]\times\SS^1
  \] 
  is well-defined and continuous. This maps provides a (free)
  homotopy between $H(0, \, \cdot) = \RR_y\circ U \circ\gamma$
  and~$H(1, \, \cdot) = u_*$, thus showing the existence 
  of a continuous lifting
  $V\colon([0, \, 1]\times Q^d)\setminus\Sigma_y\to\EE$
  of~$\RR_y\circ U$.
  Not only is~$V$ continuous, but also it is locally Lipschitz,
  because $\RR_y\circ U$ is locally Lipschitz
  on~$([0, \, 1]\times Q^d)\setminus\Sigma_y$
  and the map~$\pi$ is a local isometry.
%   In fact, the chain rule implies
%   \begin{equation} \label{smooth2V}
%    \abs{\nabla V} = \abs{\nabla(\RR_y\circ U)}
%    \qquad \textrm{a.e. on }  Q^d\setminus T_y.
%   \end{equation}

  We define~$v(x) := V(0, \, x)$ for~$x\in
%   (\{0\}\times Q^d)\setminus\Sigma_y =  
  Q^d\setminus T_y$. Then, 
  $v\colon  Q^d\setminus T_y\to\EE$ is a
  locally Lipschitz lifting of~$\RR_y\circ u$, and 
  because~$\pi$ is a local isometry, we deduce that
  \begin{equation} \label{smooth2}
   \abs{\nabla v} = \abs{\nabla(\RR_y\circ u)}
   \qquad \textrm{a.e. on }  Q^d\setminus T_y
  \end{equation}
  via the chain rule. As a consequence,
  \begin{equation} \label{locLip}
   \nabla v \in L^\infty_{\mathrm{loc}}(Q^d\setminus S_y,
   \, \R^{\ell\times d})
  \end{equation}
  because~$\RR_y\circ u$ is locally Lipschitz on~$Q^d\setminus S_y$.
  However, the distributional derivative $\D v$ of~$v$
  does not coincide with~$\nabla v$, in general; it will also 
  contain a singular part, which is carried by~$T_y$.
 \end{step}
 
 \begin{step}[Bounds on the jump of~$v$]
  Given two points~$w_1, \, w_2\in\EE$, we denote 
  by~$\dist_{\EE}(w_1, \, w_2)$ the geodesic distance between them.
%   By triangulating~$T_y$ and~$S_y$, we can write~$T_y$ as a 
%   union of $(d-1)$-simplices $H_1, \, \ldots, \, H_q$, 
%   in such a way that~$S_y\cap\inter(H_i) = \emptyset$ for any~$i$.
%   (As above, $\inter(H_i)$ denotes the interior of~$H_i$.)
%   Since~$v\in W^{1,1}(Q^d\setminus T_y)$ by~\eqref{smooth1} and~\eqref{smooth2},
%   for $\H^{d-1}$-a.e.~$x_0\in\inter(H_i)$ we have well-defined traces
%   $v^+(x_0)$, $v^-(x_0)$, on either side of~$H_i$.
%   We claim that there exists a constant~$C$, depending only
%   on~$\NN$, $m$ $\X$, $\RR$ and~$\Lambda$, such that
%   \begin{equation} \label{smooth3}
%    \dist_{\EE}(v^+(x_0), \, v^-(x_0)) \leq C
%   \end{equation}
%   for any~$i$ and $\H^d$-a.e. $x_0\in\inter(H_i)\cap Q^d$. Indeed,
%   fix~$x_0\in\inter(H_i)\cap Q^d$ and consider the line $H_i^\perp$
%   that is orthogonal to~$K_i$ and passes through~$x_0$.
  Thanks to~\eqref{locLip}, the map~$v$ has well-defined traces
  $v^+$, $v^-$ on either side of~$T_y$, $\H^{d-1}$-a.e. on~$T_y$.
  We claim that there exists a constant~$C_\Lambda$, depending only
  on~$\NN$, $m$, $\X$, $\RR$ and~$\Lambda$, such that
  \begin{equation} \label{smooth3}
   \dist_{\EE}(v^+(x), \, v^-(x)) \leq C_\Lambda 
   \qquad \textrm{for } \H^{d-1}\textrm{-a.e. } x\in T_y.
  \end{equation}
  For this purpose, take a point~$x\in T_y\setminus S_y$ that
  belongs to the interior of a $(d-1)$-polyhedron of~$T_y$.
  Let $L$ be a straight line segment
  that is orthogonal to~$T_y$ at~$x$, contains~$x$ in its interior,
  and intersects~$T_y$ only at~$x$.
  Let~$x_-$, $x_+$ be the endpoints of~$L$.
  Since~$\RR_y\circ u$ is Lispchitz continuous
  in a neighbourhood of~$x\in Q^d\setminus S_y$,
  we may take $L$ so small that 
  \begin{equation} \label{smooth3-1}
   \int_{L}\abs{\nabla(\RR_y\circ u)} \, \d\H^1\leq 1.
  \end{equation}
  We define $\gamma\colon [0, \, 4]\to\NN$, 
  $\tilde{\gamma}\colon [0, \, 4]\to\EE$ as
  \begin{gather*}
   \gamma(t) := \begin{cases}
                (\RR_y\circ u)\left((1-t)x + t x_+\right)
                       & \textrm{if } 0 \leq t \leq 1 \\
                (\RR_y\circ U)(t - 1, \, x_+)
                       & \textrm{if } 1 \leq t \leq 2 \\
                (\RR_y\circ U)(3 - t, \, x_-)
                       & \textrm{if } 2 \leq t \leq 3 \\
                (\RR_y\circ u)\left((4-t)x_- + (t-3) x\right)
                      & \textrm{if } 3 \leq t \leq 4,
               \end{cases} \\
  \tilde{\gamma}(t) := \begin{cases}
                v\left((1-t)x + t x_+\right)
                       & \textrm{if } 0 \leq t \leq 1 \\
                V(t - 1, \, x_+)
                       & \textrm{if } 1 \leq t \leq 2 \\
                V(3 - t, \, x_-)
                       & \textrm{if } 2 \leq t \leq 3 \\
                v\left((4-t)x_- + (t-3) x\right)
                      & \textrm{if } 3 \leq t \leq 4.
               \end{cases}
  \end{gather*}
  Using that $U(0, \, \cdot) = u$ and~$U(1, \, \cdot)$ is constant,
  it can be checked that $\gamma$ is indeed continuous ---
  in fact, Lipschitz. The map~$\tilde{\gamma}$ is also well-defined 
  and Lipschitz. Indeed, $V$ is continuous on~$\{1\}\times Q^d$,
  because $(\{1\}\times Q^d)\cap\Sigma_y=\emptyset$,
  and $\pi\circ V$ is constant on~$\{1\}\times Q^d$,
  so $V$ must be constant on~$\{1\}\times Q^d$, too.
  Now, $\tilde{\gamma}$ is a lifting of~$\gamma$. Since~$\pi$
  is a local isometry, we must have
  \begin{equation} \label{smooth3-2}
   \dist_{\EE}(v^{+}(x), \, v^-(x)) \leq 
   \int_0^4 \abs{\tilde{\gamma}^\prime(t)} \, \d t
   = \int_0^4 \abs{\gamma^\prime(t)} \, \d t.
%    \stackrel{\eqref{smooth3-1}}{\leq} 
%     \int_1^3 \abs{\gamma^\prime(t)} \, \d t + 1
  \end{equation}
  The maps $t\in [0, \, 1]\mapsto U(t,\ x_+)$, 
  $t\in [0, \, 1]\mapsto U(t,\ x_-)$ are well-defined and Lipschitz
  (because %$([0, \, 1]\times\{x_-, \, x_+\})\cap (U-y)^{-1}(\X)=\emptyset$
  $x_+\notin T_y$, $x_-\notin T_y$),
  and parametrise injectively straight lines segments that 
  are contained in~$Q^m_\Lambda\setminus\X$. Therefore,
  by applying Proposition~\ref{prop:X} (iv) and~\eqref{smooth3-1},
  we deduce that
  \begin{equation} \label{smooth3-3}
   \int_0^4 \abs{\gamma^\prime(t)} \, \d t 
%    \leq \int_1^3 \abs{\gamma^\prime(t)} \, \d t 
%    + \int_{[0, \, 1]\cup [3, \, 4]} \abs{\gamma^\prime(t)} \, \d t 
   \leq C_\Lambda.
  \end{equation}
  By combining~\eqref{smooth3-2} and~\eqref{smooth3-3}, 
  the claim~\eqref{smooth3} follows.
 \end{step}

 \begin{step}[Conclusion, in case~$d\geq 2$]
  From~\eqref{smooth3}, we immediately obtain
  \begin{equation} \label{smooth4}
   \abs{v^+(x) - v^-(x)} \leq C_\Lambda \qquad 
   \textrm{for } \H^1\textrm{-a.e. } x\in T_y.
  \end{equation}
  From~\eqref{smooth1}, \eqref{smooth2}, \eqref{locLip}
  and~\eqref{smooth4}, we deduce that the distributional 
  derivative of~$\D v$ is a bounded measure, with
  \begin{equation} \label{smooth5}
   \abs{\D v} \! (Q^d) \leq \norm{\nabla v}_{L^1(Q^d\setminus T_y)}
   + C_\Lambda \H^{d-1}(T_y)
   \leq C_\Lambda\norm{\nabla u}_{L^1(Q^d)} \! .
  \end{equation}
  By a Poincar\'e-type inequality in the space~BV
  (see e.g. \cite[Eq.~(16)]{Chiron-trace}), 
  there exist~$w_*\in\EE$ and a constant~$C$ 
  (depending only on~$d$) such that
  \begin{equation} \label{Poincare}
   \int_{Q^d} \dist_{\EE}(v(x), \, w_*) \, \d x \leq
   C\abs{\D v}\!(Q^d).
  \end{equation}
  By Lemma~\ref{lemma:cover}, and up to composition
  with an element of~$\Aut(\pi)$, we may assume without
  loss of generality that
  \begin{equation} \label{smooth6}
   w_*\in E_*,
  \end{equation}
  where~$E_*$ is the compact subset of~$\EE$
  given by Lemma~\ref{lemma:cover}. Now, the proposition follows 
  from \eqref{smooth5}, \eqref{Poincare} and~\eqref{smooth6}.
 \end{step}
 
 \begin{step}[The case~$d=1$]
  In case~$d=1$, Lemma~\ref{lemma:projection} 
  implies that~$\RR_y\circ u$ is continuous
  on~$[-1, \, 1]$ for a.e.~$y$,
  via Sobolev embedding. As a consequence, 
  for a.e.~$y$ the map~$\RR_y\circ u$ has a continuous
  lifting~$v\colon [-1, \, 1]\to\EE$ and actually,
  $v\in W^{1,1}(-1, \, 1; \, \EE)$ because~$\pi$ is a local isometry.
  Now Proposition~\ref{prop:affine} follows 
  from the same arguments as above.
  \qedhere
 \end{step}
\end{proof}

% \begin{remark} \label{remark:nonsharp}
%  In case~$d = 1$, the map~$\RR_\y\circ u$ for a.e.~$y\in B^m_\sigma$,
%  by Lemma~\ref{lemma:projection} and Sobolev embedding. 
%  Therefore, for a.e.~$y$ there exists a lifting~$w\in W^{1,1}(\Omega, \, \EE)$
%  with~$\abs{\nabla w}_{L^1(\Omega)} = \abs{\nabla(\RR_y\circ u)}_{L^1(\Omega)}$.
% \end{remark}

\begin{remark}
 In case~$\pi_1(\NN)$ is finite, the proof fo Proposition~\ref{prop:affine}
 simplifies considerably. Indeed, if~$\NN$ is compact and~$\pi_1(\NN)$ is finite,
 then $\EE$ is compact and hence, the estimate~\eqref{smooth3}
 is immediate. The finiteness of~$\pi_1(\NN)$ proves to be quite useful 
 in other contexts too, for instance, in the asymptotic analysis of minimisers 
 for variational problems
 % such as the Landau-de Gennes model for nematic liquid crystals
 \cite{pirla, pirla3}, or in the study of extension problems for manifold-valued maps
 \cite{Bethuel-Extension, MironescuVanSchaftingen}.
\end{remark}

\section{Proof of Theorem~\ref{th:main}}
\label{sect:main}

 Let~$\Omega\subseteq\R^d$ be a bounded, smooth domain,
 and let~$u\in\BV(\Omega, \, \NN)$.
 We first reduce to the case~$\Omega$ is a cube.
 Up to scaling, we may assume without loss of generality that
 $\overline{\Omega}\subseteq Q^d := (-1, \, 1)^d$.
 Let $u_\Omega := \L^d(\Omega)^{-1}\int_\Omega u\in\R^m$
 be the average of~$u$ over~$\Omega$.
 Thanks to \cite[Proposition~3.21]{AmbrosioFuscoPallara} and the
 BV-Poincar\'e inequality \cite[Theorem~3.44]{AmbrosioFuscoPallara},
 we can extend~$u - u_\Omega$
 to a map~$\tilde{u}\in(L^\infty\cap\BV)(Q^d, \, \R^m)$ 
 that satisfies $\abs{\D\tilde{u}}\!(Q^d)\leq C\abs{\D u}\!(\Omega)$,
 for some constant~$C$ depending only on~$\Omega$. 
 By re-defining $u := \tilde{u} + u_\Omega$, we obtain an
 extension of the map we had before. The new map belongs 
 to~$(L^\infty\cap\BV)(Q^d, \, \R^m)$ and satisfies
 \begin{equation} \label{main0}
  \abs{\D u}\!(Q^d)\leq C\abs{\D u}\!(\Omega).
 \end{equation}
 We can approximate $u$ with a sequence of smooth maps 
 $\tilde{u}_j\colon Q^d\to\R^m$ that
 converge to~$u$ weakly in~$\BV(Q^d)$, strongly in~$L^1(Q^d)$
 and a.e., and moreover
 \begin{equation} \label{main1}
  \lim_{j\to+\infty} \norm{\nabla\tilde{u}_j}_{L^1(Q^d)}
  \leq \abs{\D u} \! (Q^d)
 \end{equation}
 (see e.g. \cite[Theorem~3.9% and Remark~3.22
 ]{AmbrosioFuscoPallara}).
 By~\eqref{Lambda} and a truncation argument, we may also assume that
 \begin{equation*} %\label{main2}
  \tilde{u}_j(x)\in Q^m_\Lambda \qquad \textrm{for any }
  x\in Q^d \textrm{ and any } j.
 \end{equation*}
 Finally, for any~$j\in\N$ we may choose a piecewise-affine 
 interpolant~$u_j\colon Q^d\to Q^m_\Lambda$ 
 of~$\tilde{u}_j$ in such a way that
 \begin{equation} \label{main3}
  \norm{u_j - \tilde{u}_j}_{L^1(Q^d)}
  + \norm{\nabla u_j - \nabla \tilde{u}_j}_{L^1(Q^d)}
  \leq 1/j.
 \end{equation}
 By applying Proposition~\ref{prop:affine}, \eqref{main1} and~\eqref{main3},
 for any~$j$ we find~$y_j\in B^m_\sigma$ and a
 lifting~$v_j\in\BV(Q^d , \, \EE)$
 of~$\RR_{y_j}\circ u_j$ such that
 \begin{equation*} %\label{main4}
  \limsup_{j\to+\infty }\norm{v_j}_{L^1(Q^d)}
  \leq C_\Lambda \left(\abs{\D u} \! (\Omega) + 1\right) 
  \! , \qquad
  \limsup_{j\to+\infty }\abs{\D v_j} \! (Q^d)
  \leq C_\Lambda \abs{\D u} \! (\Omega).
 \end{equation*}
%  where the constant~$C$ depends only 
%  on~$\NN$, $\EE$, $\X$, $\RR$, $\sigma$.
 Therefore, there exist $y\in B^m_\sigma$
 and~$v\in\BV(Q^d, \, \EE)$ such that,
 up to extraction of subsequences, $y_j\to y$, 
 $v_j\rightharpoonup v$ weakly in~$\BV(Q^d)$, 
 strongly in~$L^1(Q^d)$ and a.e.~on~$Q^d$.
 (The set~$\BV(\Omega, \, \EE)$ is closed with respect to the 
 strong $L^1$-convergence, because we have embedded~$\EE$
 as a \emph{closed} subset of~$\R^{\ell}$;
 see~\cite{Muller-ClosedEmbeddings}.)
 Since~$\RR$ is smooth in a neighbourhood of~$\NN$ and~$\NN$ is compact, 
 $\RR_{y_j}\to\RR_{y}$ uniformly in a neighbourhood of~$\NN$. 
 For a.e.~$x\in\Omega$, we have~$u_j(x)\to u(x)\in\NN$
 and hence, $u_j(x)$ is arbitrarily close to~$\NN$
 for~$j$ large (depending on~$x$). As a consequence, 
 \[
  \pi\circ v_j = \RR_{y_j}\circ u_j \to \RR_{y}\circ u = u
  \qquad \textrm{a.e. on } \Omega
 \]
 and~$v_{|\Omega}$ is a lifting of~$u$.
% The total variation of~$\D v$ is bounded
% in terms of~$\D u$, due to~\eqref{main3}.
 To complete the proof of Theorem~\ref{th:main},
 it only remains to check that
 $v\in\SBV(\Omega, \, \EE)$ in case~$u\in\SBV(\Omega, \, \NN)$.
 This can be done, e.g., by repeating word 
 by word the arguments of \cite[Theorem~3, Step~4]{CO1}.
\qedhere

\bibliographystyle{plain}
\bibliography{singular_set}

\end{document}